\documentclass{article}
\usepackage[utf8]{inputenc}
\usepackage{amsthm}
\usepackage{amssymb}
\usepackage{amsmath}
\usepackage{breqn}
\usepackage[dvipsnames]{xcolor}
\usepackage{calrsfs}
\usepackage{hyperref}
\usepackage{authblk} 
\usepackage{geometry}
\usepackage{graphicx}
\usepackage{dsfont}
 \newtheorem{thm}{Theorem}[section]
 
 \newtheorem{lem}[thm]{Lemma}
 
 \theoremstyle{definition}
 
 \theoremstyle{remark}

\numberwithin{equation}{section}
\DeclareMathAlphabet{\pazocal}{OMS}{zplm}{m}{n}
\providecommand{\keywords}[1]{\textbf{\textit{Keywords: }} #1}
\providecommand{\subjclass}[1]{\textbf{\textit{Mathematics Subject Classification (2010): }} #1}
\def\blfootnote{\xdef\@thefnmark{}\@footnotetext}

\title{\bf Piezoelectric beam with magnetic effect, time-varying delay and time-varying weights}

\author[1]{\textsc{C. A. S. Nonato} \thanks{carlos.mat.nonato@hotmail.com}}
\affil[1]{\small Department of Mathematics, Federal University of Bahia, Salvador, BA,  Brazil}

\author[2]{\textsc{M. J. Dos Santos} \thanks{jeremias@ufpa.br, $^*$ Corresponding author}}
\affil[2]{\small Faculty of Exact Sciences and Technology, Federal University of Pará, Abaetetuba, PA, Brazil.}

\author[3]{\textsc{C. A. Raposo} \thanks{raposo@ufsj.edu.br}}
\affil[3]{\small Department of Mathematics, 
Federal University of São João del-Rei, MG, Brazil.}

\date{ }
\begin{document}

\maketitle	
	
\begin{abstract}
The main result of this work is to obtain the exponential decay of the solutions of a piezoelectric beam model with magnetic effect and delay term. The dampings are inserted into the equation of longitudinal displacement. The terms of damping, whose weight associated with them varies over time, are of the friction type, and one of them has delay. This work will also address the issue of existence and uniqueness of solution for the model.
\end{abstract}

\noindent\rule{\textwidth}{0.5pt}\\
\noindent\subjclass{Primary 35B40; Secondary 35Q74}\\
\noindent\keywords{Piezoelectric beam $\cdot$ Energy decay $\cdot$ magnetic effect $\cdot$ time-varying delay $\cdot$ time-varying weights}\\
\noindent\rule{\textwidth}{0.5pt}

\maketitle
\section{Introduction}

It is already known, since the 19th century that materials such as quartz, Rochelle salt and  barium titanate under pressure produces electric charge/voltage, this phenomenon is called the direct piezoelectric effect and was discovered by brothers Pierre and Jacques Curie in 1880. This same materials, when subjected to an electric field, produce proportional geometric tension. Such a phenomenon is known as the converse piezoelectric effect and was discovered by Gabriel Lippmann in 1881 \cite{doi:10.1016/B978-0-08-102135-4.00001-1}.

Currently, there is unanimity among researchers of the area, in characterizing piezoelectric materials such as those that have the physical property of transforming mechanical energy into electrical energy and vice versa \cite{isbn:9780849344596,doi:10.1007/978-1-4899-6453-3,doi:10.1007/b101799}, more precisely, they undergo mechanical deformations when placed in an electric field and under mechanical loads they become electrically polarized. This type of special property has great applicability in the modern industry, therefore, these materials have been widely used in the production of electromechanical devices, such as sensors for data collection, transducers for converting electric energy to mechanical energy or vice versa, resonators for timekeeping and telecommunication and actuators.

The ability of piezoelectric structures to generate deformations controlled by electrical field applications and vice versa, attracts the attention of scientists from various areas in order to design mathematical and computational models capable of providing new knowledge and applications of these materials. In particular, Mathematics can provide several tools for studying the solution behavior of piezoelectric beam models, including Numerical Analysis, Dynamic Systems, Controllability and Stabilization.

Due to the fact that magnetic energy has a relatively small effect on the general dynamics, magnetic effects are neglected in piezoelectric beam models. However, in closed loop, the magnetic effect can cause oscillations in the output which results in the instability of the system, this tells us that the magnetic effect can cause a limitation in the performance of the system \cite{doi:10.1063/1.1139566,doi:10.1016/j.simpat.2007.11.005}.

In many studies related to piezoelectric structures, the magnetic effect is neglected and only the mechanical and electrical effects are considered. In general the mechanical effects are modeled using Kirchhoff, Euler-Bernoulli or Mindlin-Timoshenko assumptions for small displacements \cite{isbn:9780471970248,doi:10.1109/CDC.1998.757931,doi:10.1137/1.9780898717471,doi:10.1007/b101799} and electrical and magnetic effects are added to the system generally using electrostatic, quasi-static and fully dynamic approaches \cite{doi:10.1007/978-1-4899-6453-3}. The electrostatic and quasi-static approaches (see for example \cite{DESTUYNDER-1992,doi:10.1109/CDC.1998.757931, doi:10.1137/050629884,doi:doi.org/10.1016/j.crma.2008.12.007,isbn:9780849344596,doi:10.1137/1.9780898717471,doi:10.1007/978-1-4899-6453-3,isbn:9789402412581}), despite being widely used, completely exclude the magnetic effect as well as its couplings with mechanical and electrical effects.

Morris and Özer in \cite{doi:10.1109/CDC.2013.6760341,doi:10.1137/130918319}, proposed, a variational approach, a piezoelectric beam model with a magnetic effect, based on the Euler-Bernoulli and Rayleigh beam theory for small displacement (the same equations for the model are obtained if Mindlin-Timoshenko small displacement assumptions are used \cite{doi:10.1109/ACC.2014.6858862}), they considered an elastic beam covered by a piezoelectric material on its upper and lower surfaces, isolated at the edges and connected to a external electrical circuit to feed charge to the electrodes. As the voltage is prescribed at the electrodes, the following Lagrangean was considered
\begin{eqnarray}
\pazocal{L}=\int_0^T[\mathbf{K}-(\mathbf{P}+\mathbf{E})+\mathbf{B}+\mathbf{W}]\,dt,\label{1:1}
\end{eqnarray}
where $\mathbf{K}$, $\mathbf{P}+\mathbf{E}$, $\mathbf{B}$ and $\mathbf{B}$ represent the (mechanical) kinetic energy, total stored energy, magnectic energy (electrical kinetic) of the beam and the work done by external forces, respectively. For a beam of length $L$ to thickness $h$ and considering $v=v(x,t)$, $w=w(x,t)$  and $p=p(x,t)$ as functions that represent the longitudinal displacement of the center line, transverse displacement of the beam and the total load of
the electric displacement along the transverse direction at each point $x$, respectively. So, one can assume that
\begin{eqnarray}
\begin{aligned}
\mathbf{P}+\mathbf{E}=\frac{h}{2}\int_0^L \left[ \alpha \left( v_x^2+\frac{h^2}{12}w^2_{xx}-2\gamma\beta v_xp_x+\beta p_x^2 \right) \right]\,dx, \quad \mathbf{B}=\frac{\mu h}{2}\int_0^Lp_t^2\,dx, \\
\mathbf{K}=\frac{\rho h}{2}\int_0^L \left( v_t^2+\frac{h^2}{12}w_t^2+w_t^2 \right]\,dx \quad \text{and}\quad \mathbf{W}=-\int_0^L p_xV(t)\,dx,
\end{aligned}\label{1:2}
\end{eqnarray}
where $V(t)$ is the voltage applied at the electrode. From Hamilton's principle for admissible displacement variations $\{v,w,p\}$ of $L$ the zero and observing that the only external force acting on the beam is the voltage at the electrodes (the bending equation is decoupled), they got the system
\begin{eqnarray}
\begin{aligned}
\rho v_{tt}-\alpha v_{xx}+\gamma \beta p_{xx}=&\ 0,\\
\mu p_{tt}-\beta p_{xx}+\gamma\beta v_{xx}=&\ 0.
\end{aligned}\label{1:3}
\end{eqnarray}
where $\rho$, $\alpha$, $\gamma$, $\mu$ and $\beta$ denote the mass density, elastic stiffness, piezoelectric coefficient, magnetic permeability, water resistance coefficient of the beam and the prescribed voltage on electrodes of beam respectively, and in addition, the relationship is considered
\begin{eqnarray}
\alpha =\alpha_1+\gamma^2\beta,\label{1:4}
\end{eqnarray}
they assumed that the beam is fixed at $x=0$ and free at $x=L$, and thus they got (from modeling) the following boundary conditions
\begin{eqnarray}
\begin{aligned}
v(0,t)=\alpha v_x(L,t)-\gamma\beta p_x(L,t)=&\ 0, \\
p(0,t)=\beta p_x(L,t)-\gamma\beta v_x(L,t)=&\ -\frac{V(t)}{h}.
\end{aligned}\label{1:5}
\end{eqnarray}
Then, the authors consider $V(t)=kp_t(L,t)$ (electrical feedback controller) in \eqref{1:5}  and establish strong stabilization for almost all system parameters and exponential stability for system parameters in a null measure set.

It is worth mentioning that it is well known that piezoelectric beams without the magnetic effect, in which they are represented by a wave equation \cite{doi:10.1137/130918319}, are exactly observable \cite{doi:10.1137/1030001} and exponentially stable \cite{doi:10.1007/s10444-004-7629-9}.

Ramos et al. in \cite{doi:10.1051/m2an/2018004} inserted a (mechanical) dissipative term $\delta v_t$ in $\eqref{1:3}_1$, where $\alpha>0$ is a constant and considered the following boundary condition
\begin{eqnarray}
\begin{aligned}
v(0,t)=\alpha v_x(L,t)-\gamma\beta p_x(L,t)=&\ 0, \\
p(0,t)=\beta p_x(L,t)-\gamma\beta v_x(L,t)=&\ 0.
\end{aligned}\label{1:7}
\end{eqnarray}
Note that, using \eqref{1:4}, condition \eqref{1:7} is equivalent to Dirichlet-Neumann condition
\begin{eqnarray}
v(0,t)=p(0,t)=v_x(L,t)=p_x(L,t)=0.
\end{eqnarray}
The authors showed, by using energy method, that the system's energy decays exponentially. This means that the friction term and the magnetic effect work together in order to exponentially stabilizes the system.

Ramos et al. in \cite{doi:10.1007/s00033-019-1106-2} considered the piezoelectric beam with magnetic effect \eqref{1:3} with boundary conditions given by
\begin{eqnarray}
\begin{aligned}
v(0,t)=\alpha v_x(L,t)-\gamma\beta p_x(L,t)+\xi_1\frac{v_t(L,t)}{h}=&\ 0, \\
p(0,t)=\beta p_x(L,t)-\gamma\beta v_x(L,t)+\xi_2\frac{p_t(L,t)}{h}=&\ 0.
\end{aligned}\label{1:8}
\end{eqnarray}
They showed that the system is exponentially stable regardless of any relationship between system parameters and exponential stability is equivalent to exact observability at the boundary.

In parallel to this, due to technological advances in the production and design of precision control mechanisms such as sensors and actuators, there was an increasing need to study the effects of information delay in order to improve the performance and control of these devices \cite{doi:10.1007/978-3-662-05030-9}. Delay effects are present in almost every real mechanical system and in most situations such an effect is inevitable. Therefore, models that take into account the effect of delay are more realistic \cite{doi:10.1007/978-1-4939-2107-2}.

An important example of the delay effect is the active control of civil engineering structures in which different control systems are installed in tall buildings. The control process involves several steps such as vibrational data measurement, filtering and conditioning of this data, computing control forces, transmission of data and signals to actuators, application of control forces necessary to the structure. If force applications are not synchronized due to time delay, it can make the structure unstable \cite{doi:10.1002/(SICI)1096-9845(199711)26:11<1169::AID-EQE702>3.0.CO;2-S}.

In the context of models consisting of partial differential equations, when we insert delay feedback terms into models that were stable, they can become unstable \cite{doi:10.1109/CDC.1985.268529,doi:10.1137/0326040,doi:10.1137/060648891}. Therefore, for these types of models (formed by partial differential equations) we should be careful to analyze each case.

Nicaise et al. in \cite{doi:10.3934/dcdss.2011.4.693}, studied the following wave equation with boundary time-varying delay
\begin{eqnarray}
\begin{aligned}
u_{tt}-\Delta u=&\ 0\quad\text{in}\quad \Omega\times(0,\infty), \\
u=&\ 0 \quad\text{in}\quad \Gamma _D\times(0,\infty),\\
\frac{\partial u}{\partial \nu}=-\mu_1 u_t-\mu_2 u_t(x,t-\tau(t))&\ 0 \quad\text{in}\quad \Gamma _N\times(0,\infty),\\
u(x,0)=u_0(x)\quad\text{and}\quad u_t(x,0)=u_1(x)&\quad \ \ \,  \text{in}\quad\Omega,\\
u_t(x,t-\tau(0))=f_0(x,t-\tau(0))&\quad \ \ \,  \text{in}\quad\Gamma_N\times(0,\tau(0)),
\end{aligned}
\end{eqnarray}
where $\Omega\subset\mathds{R}^n$ is bounded and smooth domain, $\mu_1$ and $\mu_2$  are positive constants, $\nu(x)$ represent the outer unit normal vector to the point $x\in\Gamma$ and $\frac{\partial u}{\partial\nu}$ is the normal derivative, $\Gamma=\Gamma_D\cup \Gamma_N$ is the boundary of $\Omega$.  At work, it was considered
\begin{eqnarray}
\tau\in W^{2,\infty}([0,T]),\quad \forall T>0,\label{1:9} \\
0<\tau_0\leq \tau(t)\leq \overline{\tau},\quad\forall t>0,\label{1:10}
\end{eqnarray}
for some constants $\tau_0$ and $\overline{\tau}$ and there exists $d>0$ such that
\begin{eqnarray}
\mu_2<\sqrt{1-d}\mu_1\label{1:11}
\end{eqnarray}
with
\begin{eqnarray}
\tau'(t)\leq d<1, \quad\forall t>0. \label{1:12}
\end{eqnarray}
With these assumptions, the authors showed that the system is exponentially stable.

Kirane et al. in \cite{doi:10.3934/cpaa.2011.10.667}, considered the following one-dimensional Timoshenko beam model with variable delay $\tau(t)$ in the rotation angle equation
\begin{eqnarray}
\begin{aligned}
\rho_1\varphi_{tt}-\kappa(\varphi_x+\psi)_x=&\ 0\ \quad\text{in}\quad (0,1)\times(0,\infty),\\
\rho_2\psi_{tt}-b\psi_{xx}+\kappa(\varphi_x+\psi )+\mu_1 \psi_t+\mu_2 \psi_t(x,t-\tau(t))=&\ 0\ \quad\text{in}\quad (0,1)\times(0,\infty),
\end{aligned}\label{1:13}
\end{eqnarray}
where $\rho_1$, $\rho_2$, $\kappa$ and $b$ are positive constants related to the beam's physical properties, the delay function $\tau(t)$ satisfies \eqref{1:9}, \eqref{1:10} and \eqref{1:12}. The authors showed that if \eqref{1:11} and $\rho_1/\kappa=\rho_2/b$ holds, then the system is exponentially stable.

Benaissa et al. in \cite{doi:10.14232/ejqtde.2014.1.11} considered the following wave equation with delay and damping weights depending on the time
\begin{eqnarray}
\begin{aligned}
u_{tt}-\Delta u+\mu_1(t)u_t+\mu_2(t)u_t(x,t-\tau)= &\ 0\quad\text{in}\quad \Omega\times(0,\infty),\\
u=&\ 0 \quad\text{in}\quad \Gamma\times(0,\infty),\\
u(x,0)=u_0(x)\quad\text{and}\quad u_t(x,0)=u_1(x)&\quad \ \ \,  \text{in}\quad\Omega,\\
u_t(x,t-\tau(0))=f_0(x,t-\tau(0))&\quad \ \ \,  \text{in}\quad\Omega\times(0,\tau(0)),
\end{aligned}\label{1:14}
\end{eqnarray}
where $\Omega\subset\mathds{R}^n$ is a bouded domain with boundary $\Gamma$. Unlike previous works, the dampings $\mu_1$ and $\mu_2$ depend on the time $t$, however the delay time $\tau$ is constant. Under appropriate assumptions about the weights of the damping $\mu_1$ and $\mu_2$ the authors obtained the exponential decay of the energy of the system.

Barros et al. in \cite{10.3934/era.2020014} studied the problem \eqref{1:14} with $\Omega=(0,L)\subset\mathds{R}$ and $\tau=\tau(t)$ a function of time $t$. Under appropriate assumptions for $\mu_1(t)$ and $\mu_2(t)$ and considering \eqref{1:9}, \eqref{1:10} and \eqref{1:12} the authors showed that the energy of the system decays exponentially.

Our intention in mentioning the last three works was to show situations in which time-dependent delay feedback appears $\tau =\tau(t)$ as well as to show situations in which the weight of the damping may vary, which make the problem worse, undoubtedly more attractive and challenging.

There are numerous studies on exponential stability of linear systems considering the case where the delay is constant \cite{doi:10.1080/00036811.2014.1000314,doi:10.1186/s13661-015-0468-4,doi:10.1007/s00033-011-0145-0,doi:10.1137/060648891,NICAISE-PIGNOTTI-2011,doi:10.1016/j.jmaa.2018.06.017,RAPOSO-CHUQUIPOMA-AVILA-SANTOS-2013,doi:10.1016/j.amc.2010.08.021,doi:10.1051/cocv:2006021}. There are also several studies considering non-linear models with delay where the existence of attractors is investigated, among them, Timoshenko systems \cite{doi:10.1063/5.0006680,doi:10.1080/00036811.2016.1148139,doi:10.1007/s10884-019-09799-2,doi:10.1007/s00245-018-9539-0}, poroelastic systems \cite{doi:10.3934/dcdsb.2020206} and suspension bridge \cite{doi:10.1007/s00033-018-0934-9,doi:10.1186/s13660-019-2133-4}.

Based on the work mentioned above about piezoelectric beam and delay feedback, we design and propose to study and question the exponential stability for the following system
\begin{eqnarray}
\begin{aligned}
\rho v_{tt} - \alpha v_{xx} + \gamma \beta p_{xx} + \mu_1(t) v_t + \mu_2(t) v_t(x, t - \tau(t))=&\ 0\quad \text{in}\quad (0,L)\times (0,\infty), \\
\mu p_{tt} - \beta p_{xx} + \gamma \beta v_{xx} =&\  0\quad \text{in}\quad(0,L)\times (0,\infty),
\end{aligned}\label{1:15}
\end{eqnarray}
with boundary conditions given by \eqref{1:7} and initial conditions
\begin{eqnarray}
\begin{aligned}
v(x,0)=v_0(x),\quad v_t(x,0)=v_1(x), \quad p(x,0)=p_0(x),\quad  p_t(x,0)=p_1(x),&\quad  x \in (0,L), \\
v_t\left( x,-s\tau(0) \right) = v_2( x,s),&\quad (x,s) \in (0,L) \times (0,1),
\end{aligned}\label{1:16}
\end{eqnarray}
where $v_0$, $v_1$, $v_2$, $p_0$, $p_1$ are known functions belonging to appropriate functional spaces.

In \eqref{1:15} we are admitting that the delay is being considered in the longitudinal displacement of the beam, this seems to us very natural since there are several studies on piezoelectric structures considering the effect of delay on the mechanical part of the system \cite{doi:10.1080/15376490802666310,doi:10.1177/0263092316628255,doi:10.3390/app9081557,doi:10.1088/1361-665x/ab2e3d}.

We will use the standard multiplicative method to obtain the main result. The novelty of the work is found, basically, in the application of this technique in a relatively new model (piezoelectric beam with magnetic effect).

The article is organized as follows: in section \ref{sec:2}, we will consider the assumptions for the functions present in \eqref{1:15} as well as, through a change of variable, obtain a system equivalent to \eqref{1:15}. In section \ref{sec:3}, using the semigroup theory of linear operators found in \cite{isnb:978-3-642-11105-1-KATO}, the question of the existence, uniqueness and regularity of the solution will be addressed. In section \ref{sec:4}, we will obtain the main result of this work, which is the proof of the exponential decay for the system (1.1).

\section{Preliminaries and Assumptions}\label{sec:2}

In this work we will consider the following assumptions:

\begin{description}
\item[(A1)] The delay function $\tau=\tau(t)$, satisfies
\begin{equation}
\tau \in W^{2,\infty}([0,T]), \quad \forall T > 0,\label{2:1}
\end{equation}
there exist positive constants $\tau_0$, $\tau_1$ and $d$, satisfying
\begin{equation}
0<\tau_0\leq \tau(t)\leq \tau_1, \quad  \forall t>0\label{2:2}
\end{equation}
and
\begin{equation}
\tau'(t)\leq d < 1, \quad \forall t >0;\label{2:3}
\end{equation}

\item[(A2)]  $\mu_1:\mathds{R}_+ \rightarrow (0,+\infty)$ is a non-increasing function of class $C^1(\mathds{R}_+)$. In addition, there exists a constant $M_1>0$, such that
\begin{equation}
\left| \frac{\mu'_1(t)}{\mu_1(t)} \right| \leq M_1, \quad \forall t \geq 0;\label{2:4}
\end{equation}

\item[(A3)] $\mu_2:\mathds{R}_+ \rightarrow \mathds{R}$ is a function of class $C^1(\mathds{R}_+)$,which is not necessarily positives or monotones. In addition, there exist constants $M_2>0$ and $\delta$, with $0 < \delta < \sqrt{1-d}$, such that
\begin{equation}
| \mu_2(t)| \leq \delta \mu_1(t),\label{2:5}
\end{equation}
and
\begin{equation}
| \mu'_2(t)| \leq M_2 \mu_1(t).\label{2:6}
\end{equation}
\end{description}

Let us now consider the following procedure that can be found in \cite{doi:10.3934/dcdss.2011.4.693}, in order to obtain a new (independent) variable
\begin{equation}
z(x,y,t) = v_t(x,t-\tau(t) y),\quad (x,y,t)\in (0,L)\times(0,1)\times(0,\infty).\label{2:7}
\end{equation}
It is easily verified that the $z$ satisfies
\begin{eqnarray}
\tau(t)z_t(x,y,t) + (1-\tau'(t)y)z_y(x,y,t) = 0.\label{2:8}
\end{eqnarray}
Therefore, by using \eqref{2:7} and \eqref{2:8} we can rewrite \eqref{1:15} as follows

\begin{eqnarray}
\begin{aligned}
\rho v_{tt} - \alpha v_{xx} + \gamma \beta p_{xx} + \mu_1(t) v_t + \mu_2(t) z(x,1, t)=&\ 0\quad \text{in}\quad (0,L)\times (0,\infty), \\
\mu p_{tt} - \beta p_{xx} + \gamma \beta v_{xx} =&\  0\quad \text{in}\quad(0,L)\times (0,\infty),\\
\tau(t)z_t + (1-\tau'(t)y)z_y =&\ 0\quad \text{in}\quad(0,L)\times(0,1)\times (0,\infty),
\end{aligned}\label{2:9}
\end{eqnarray}
subject to boundary conditions given in \eqref{1:7}, that is,
\begin{eqnarray}
v(0,t)=p(0,t)=v_x(L,t)=p_x(L,t)=0\label{2:10}
\end{eqnarray}
and initial conditions
\begin{eqnarray}
\begin{aligned}
v(x,0)=v_0(x),\quad v_t(x,0)=v_1(x), \quad p(x,0)=p_0(x),\quad  p_t(x,0)=p_1(x),&\quad\text{in}\quad (0,L), \\
z(x,y,0) = v_2(x,y),&\quad\text{in}\quad (0,L) \times (0,1).
\end{aligned}\label{2:11}
\end{eqnarray}

\section{Well-posedness}\label{sec:3}

In this section, using the theory of semigroups of linear operators found in \cite{isnb:978-3-642-11105-1-KATO}, a result of existence, uniqueness and regularity will be obtained for the problem \eqref{2:9}-\eqref{2:11}. Similar procedures are found in \cite{doi:10.3934/cpaa.2011.10.667,doi:10.1007/s00161-017-0556-z,doi:10.3934/dcdss.2011.4.693}.

Firstly, consider the following spaces
\begin{eqnarray}
H_*(0,L) = \{ \eta \in H^1(0,L); \ \eta(0)=0 \}\label{3:1}
\end{eqnarray}
and
\begin{eqnarray}
\pazocal{H} = H_*(0,L) \times L^2(0,L) \times H_*(0,L) \times L^2(0,L) \times L^2((0,L)\times(0,1)).\label{3:2}
\end{eqnarray}
We define on $\pazocal{H}$ the following inner product

\begin{eqnarray}
\begin{aligned}
\langle U, \tilde{U} \rangle_{\pazocal{H}} = &\ \rho \int_0^L u \widetilde{u}\,dx + \mu \int_0^L q \widetilde{q}\,dx + \alpha_1 \int_0^L v_x \tilde{v}_x\,dx \\
& + \beta \int_0^L ( \gamma v_x - p_x )( \gamma \widetilde{v}_x - \widetilde{p}_x )\,dx + \int_0^L \int_0^1 z \tilde{z}\,dy\,dx,
\end{aligned}\label{3:3}
\end{eqnarray}
for any $U = (v,u,p,q,z )$, $\widetilde{U} = (\widetilde{v},\widetilde{u},\widetilde{p},\widetilde{q},\widetilde{z})$ in $\pazocal{H}$.

Introducing $U(t)=(v(t),v_t(t),p(t),p_t(t),z(t))^T$ and $U_0=(v_0,v_1,p_0,p_1,v_2)^T$,  the system \eqref{2:9}-\eqref{2:11} can be written as the following abstract initial value problem in $\pazocal{H}$
\begin{eqnarray}
\left\{\begin{array}{rcl}
U_{t}(t) &=&\pazocal{A}(t) U(t) ,\quad t>0, \\
U(0) &=& U_0,
\end{array}
\right.\label{pc}
\end{eqnarray}
where the operator $\pazocal{A}(t):D(\pazocal{A}(t))\subset\pazocal{H}\to\pazocal{H}$ is given by
\begin{eqnarray}
\pazocal{A}(t)\left(\begin{array}{c}
v\\ u\\ p\\ q\\ z
\end{array}\right) = \left(
\begin{array}{c}
u \\
\rho^{-1} ( \alpha v_{xx} - \gamma \beta p_{xx} - \mu_1(t)u - \mu_2(t)z(\cdot,1) ) \\ q \\ \mu^{-1} (\beta p_{xx}-\gamma\beta v_{xx})\\ -\frac{1-\tau'(t)y}{\tau(t)} z_y
\end{array}
\right),\label{3:4}
\end{eqnarray}
with
\begin{eqnarray}
\begin{aligned}
D(\pazocal{A}(t)) = \{ &(v,u,p,q,z) \in \pazocal{H}; \ v,p\in H^2(0,L);\ u,q\in H_*(0,L); \\
& z\in L^2(0,1;H_0^1(0,L)),\ z(\cdot,0)=u\}.
\end{aligned}
\end{eqnarray}
Note that $D(\pazocal{A}(t))$ is independent of $t$, that is,
\begin{eqnarray}
D(\pazocal{A}(t))=D(\pazocal{A}(0)),\quad \forall t>0.\label{3:5}
\end{eqnarray}

A general theory for not autonomous operators given by equations of type \eqref{pc} has been developed using semigroup theory, see \cite{isbn:978-88-7642-248-5,isnb:978-3-642-11105-1-KATO,doi:10.1007/978-1-4612-5561-1}. The simplest way to prove existence and uniqueness results is to show that the triplet $\{ (\pazocal{A}, \pazocal{H},Y)\}$, with $\pazocal{A} = \{ \pazocal{A}(t);\  t \in [0,T] \}$, for some fixed $T>0$ and $Y=\pazocal{A}(0)$, forms a CD-systems (Constant Domain system, see \cite{isbn:978-88-7642-248-5,isnb:978-3-642-11105-1-KATO}). More precisely, the following theorem, which is due to Tosio Kato (Theorem 1.9 in \cite{isnb:978-3-642-11105-1-KATO}) gives the existence and uniqueness results:

\begin{thm}\label{thm:3:1}
Assume that
\begin{enumerate}
\item[(i)] $Y=D(\pazocal{A}(0))$ is  dense in $\pazocal{H}$;
\item[(ii)]$D(\pazocal{A}(t))=D(\pazocal{A}(0))$, $\forall t>0$;
\item[(iii)] for all $t \in [0,T]$, $\pazocal{A}(t)$ generates a strongly continuous semigroup on $\pazocal{H}$ and the family $\pazocal{A}= \{ \pazocal{A}(t);\ t \in [0,T] \}$ is stable with stability constants $C$ and $m$ independent of $t$ (i.e., the semigroup $(S_t(s))_{s\geq 0}$ generated by $\pazocal{A}(t)$ satisfies $\| S_t(s)W \|_{\pazocal{H}} \leq Ce^{ms} \| W \|_{\pazocal{H}}$, for all $W \in \pazocal{H}$ and $s\geq 0$);
\item[(iv)] $\partial_t \pazocal{A}(t)$ belongs to $L_{*}^{\infty}([0,T],B(Y, \pazocal{H}))$, which is the space of equivalent classes of essentially bounded, strongly measurable functions from $[0,T]$ into the set $B(Y, \mathcal{H})$ of bounded operators from $Y$ into $\pazocal{H}$.
		
Then, problem \eqref{pc} has a unique solution
\begin{eqnarray}
U \in C([0,T);Y) \cap C^1([0,T); \pazocal{H}),\label{3:6}
\end{eqnarray}
for any initial datum in $Y$.
\end{enumerate}
\end{thm}

In this way, we are ready to state and prove the main result of this section, which is

\begin{thm}[{\bf Global solution}]\label{thm:3:2}
For any $U_0 \in D(\pazocal{A}(0))$, there exists a unique solution $U$ of \eqref{pc} satisfying
\begin{eqnarray}
U \in C([0,+\infty); D(\pazocal{A}(0))) \cap C^1([0,+\infty); \pazocal{H}).
\end{eqnarray}
\end{thm}

\begin{proof}
We must show that $\pazocal{A}(t)$ meets the conditions of Theorem \ref{thm:3:1}. In fact,

\hspace{\oddsidemargin}{\bf (i)} this condition can be proven using arguments analogous to those found in \cite{doi:10.3934/cpaa.2011.10.667,doi:10.1007/s00161-017-0556-z,NICAISE-PIGNOTTI-2011,doi:10.3934/dcdss.2011.4.693}.\\

\hspace{\oddsidemargin}{\bf (ii)} It has been observed in \eqref{3:5}.\\

\hspace{\oddsidemargin}{\bf (iii)} In order to show that the operator $\pazocal{A}(t)$ generates a $C_0$-semigroup on $\pazocal{H}$, given $t$, we introduced the time-dependent inner product on $\pazocal{H}$ (this internal product is equivalent to \eqref{3:3})

\begin{eqnarray}
\begin{aligned}
\langle U, \tilde{U} \rangle_t = &\ \rho \int_0^L u \widetilde{u}\,dx + \mu \int_0^L q \widetilde{q}\,dx + \alpha_1 \int_0^L v_x \tilde{v}_x\,dx \\
&\ + \beta \int_0^L ( \gamma v_x - p_x )( \gamma \widetilde{v}_x - \widetilde{p}_x )\,dx + \xi(t)\tau(t)\int_0^L \int_0^1 z \tilde{z}\,dy\,dx,
\end{aligned}\label{3:7}
\end{eqnarray}
for any $U = (v,u,p,q,z )$, $\widetilde{U} = (\widetilde{v},\widetilde{u},\widetilde{p},\widetilde{q},\widetilde{z})$ in $\pazocal{H}$, where
\begin{eqnarray}
\xi(t)=\overline{\xi}\mu_1(t)\label{3:7:1}
\end{eqnarray}
and $\overline{\xi}$ is a positive constant such that
\begin{eqnarray}
\frac{\delta}{\sqrt{1-d}} < \bar{\xi} < 2 - \frac{\delta}{\sqrt{1-d}}.\label{3:8}
\end{eqnarray}
Note that
\begin{eqnarray}
\begin{aligned}
\langle \pazocal{A}(t)U, U \rangle_t = &\ - \mu_1(t) \int_0^L u^2\,dx - \mu_2(t) \int_0^L z(x,1)u\,dx \\
& - \frac{\xi(t)}{2} \int_0^L \int_0^1 (1-\tau'(t)y)\frac{\partial}{\partial y} z^2(x,y)\,dy\,dx,
\end{aligned}\label{3:9}
\end{eqnarray}
for any $U=(v,u,p,q,z)^T \in D(\pazocal{A}(t))$. Since
\begin{eqnarray}
(1- \tau'(t)y )\frac{\partial}{\partial y} z^2 = \frac{\partial}{\partial y} \left( (1- \tau'(t)y) z^2 \right) + \tau'(t)z^2,\label{3:10}
\end{eqnarray}
from \eqref{3:9} and \eqref{3:10} we have
\begin{eqnarray}
\begin{aligned}
\langle \pazocal{A}(t)U, U \rangle_t = & - \mu_1(t) \int_0^L u^2\,dx - \mu_2(t) \int_0^L z(x,1)u\,dx + \frac{\xi(t)}{2} \int_0^L u^2\,dx \\
& - \frac{\xi(t)(1-\tau'(t))}{2} \int_0^L z^2(x,1)\,dx - \frac{\xi(t)\tau'(t)}{2} \int_0^L \int_0^1 z^2\,dy\,dx.
\end{aligned}\label{3:11}
\end{eqnarray}
Now, applying Young's inequality to the second term on the right side of \eqref{3:11}, we get

\begin{eqnarray}
\begin{aligned}
\langle \pazocal{A}(t)U, U \rangle_t \leq &  - \left( \mu_1(t) - \frac{\xi(t)}{2}- \frac{|\mu_2(t)|}{2\sqrt{1-d}} \right) \int_0^L u^2\,dx \\
& - \left( \frac{\xi(t)}{2} - \frac{\xi(t) \tau'(t)}{2} - \frac{| \mu_2(t)| \sqrt{1-d}}{2} \right)\int_0^L z^2(x,1)\,dx \\
& + \frac{\xi(t)|\tau'(t)|}{2\tau(t)}\tau(t) \int_0^L \int_0^1 z^2\,dy\,dx.
\end{aligned}\label{3:12}
\end{eqnarray}

From {\bf (A3)} and \eqref{3:7:1}, we obtain

\begin{eqnarray}
\begin{aligned}
\langle \pazocal{A}(t)U, U \rangle_t \leq & - \mu_1(t) \left( 1-\frac{\overline{\xi}}{2}- \frac{\delta}{2\sqrt{1-d}} \right) \int_0^L u^2\,dx \\
& - \mu_1(t) \left( \frac{\overline{\xi}(1-\tau'(t))}{2}- \frac{\delta \sqrt{1-d}}{2} \right)\int_0^L z^2(x,1)\,dx \\
& + \kappa(t) \langle U, U \rangle_t,
\end{aligned}\label{3:13}
\end{eqnarray}
where
\begin{eqnarray}
\kappa(t) = \frac{\sqrt{1+\tau'(t)^2}}{2\tau(t)}.\label{3:14}
\end{eqnarray}
From  \eqref{2:2}, \eqref{2:3} and \eqref{3:8}, we have
\begin{eqnarray}
1-\frac{\overline{\xi}}{2}- \frac{\delta}{2\sqrt{1-d}}>0 \quad \mbox{and} \quad \frac{\overline{\xi}(1-\tau'(t))}{2}- \frac{\delta \sqrt{1-d}}{2} >0.\label{3:15}
\end{eqnarray}
Therefore we conclude that
\begin{equation}
\langle \pazocal{A}(t)U, U \rangle_t - \kappa(t)\langle U,U \rangle_t \leq 0,\label{3:16}
\end{equation}
which means that operator $\widetilde{\pazocal{A}}(t) = \pazocal{A}(t) - \kappa(t) I$ is dissipative (in the next steps we will use $\widetilde{\pazocal{A}}$ as a pivot to then recover the intended properties of $\pazocal{A}$).

Now, we will prove the surjectivity of the operator $\lambda I - \pazocal{A}(t)$, for fixed $t>0$. For this purpose, given $F=(f_1,f_2,f_3,f_4,f_5)^T \in \pazocal{H}$, we seek $U=(v,u,p,q,z)^T \in D(\pazocal{A}(t))$ which is solution of
\begin{eqnarray}
(\lambda I - \pazocal{A}(t))U = F,\label{3:17}
\end{eqnarray}
that is, the entries of $U$ satisfy the system of equations
\begin{eqnarray}
\lambda v - u &=& f_1, \label{3:17:1}\\
\lambda \rho u - \alpha v_{xx} + \gamma \beta p_{xx} + \mu_1(t)u + \mu_2(t) z(x,1) &=& \rho f_2, \label{3:17:2}\\
\lambda p - q &=& f_3, \label{3:17:3}\\
\lambda \mu q - \beta p_{xx} +\gamma \beta v_{xx} &=& \mu f_4, \label{3:17:4}\\
\lambda \tau(t)z + (1 - \tau'(t)y ) z_y &=& \tau(t)f_5.\label{3:17:5}
\end{eqnarray}
Suppose that we have found $v$ and $p$ with the appropriated regularity. Therefore, from \eqref{3:17:1} and \eqref{3:17:3} we have
\begin{eqnarray}
u = \lambda v - f_1,\label{3:17:6} \\
q = \lambda p - f_3,\label{3:17:7}
\end{eqnarray}
it is clear that $u,q\in H_*(0,L)$. Furthermore, if $\tau'(t)=0$, then
\begin{eqnarray}
z(x,y) = u(x)e^{-\lambda \tau(t)y} + \tau(t)e^{-\lambda \tau(t)y} \int_0^y f_5(x,s)e^{\lambda \tau(t)s}\,ds\label{3:17:9}
\end{eqnarray}
is solution of \eqref{3:17:5} satisfying
\begin{eqnarray}
z(x,0)=u(x).\label{3:17:10}
\end{eqnarray}
Otherwise,
\begin{eqnarray}
z(x,y) = u(x)e^{\sigma(y,t)} +\tau(t) e^{\sigma(y,t)} \int_0^y \frac{f_5(x,s)}{1-\tau'(t)s} e^{-\sigma(s,t)}\,ds,
\end{eqnarray}
where
\begin{eqnarray}
\sigma(y,t)=\frac{\lambda \tau(t)}{\tau'(t)}\ln(1-\tau'(t)y),
\end{eqnarray}
is solution of \eqref{3:17:5} satisfying \eqref{3:17:10}. From now on, for practicality purposes, we will consider $\tau'(t)=0$ (the case $\tau(t)\neq 0$ is analogous), this way we have (taking into account \eqref{3:17:6})
\begin{eqnarray}
\begin{aligned}
z(x,1) =&\ ue^{-\lambda \tau(t)} + \tau(t)e^{-\lambda \tau(t)} \int_0^1 f_5(x,s)e^{\lambda \tau(t)s}\,ds\\
=&\  (\lambda v-f_1)e^{-\lambda \tau(t)} + \tau(t)e^{-\lambda \tau(t)} \int_0^1 f_5(x,s)e^{\lambda \tau(t)s}\,ds \\
=&\ \lambda ve^{-\lambda \tau(t)}-f_1e^{-\lambda \tau(t)} + \tau(t)e^{-\lambda \tau(t)} \int_0^1 f_5(x,s)e^{\lambda \tau(t)s}\,ds.
\end{aligned}\label{3:17:11}
\end{eqnarray}
Substituting \eqref{3:17:6} and \eqref{3:17:11} in \eqref{3:17:2}, and \eqref{3:17:7} in \eqref{3:17:4}, we obtain
\begin{eqnarray}
\begin{aligned}
\eta v - \alpha v_{xx} + \gamma \beta p_{xx} &=&g_1, \\
\lambda^2 \mu p  - \beta p_{xx} +\gamma \beta v_{xx} &=& g_2,
\end{aligned}\label{3:17:12}
\end{eqnarray}
where
\begin{eqnarray}
\begin{aligned}
\eta:=&\ \lambda^2 \rho +\lambda \mu_1(t)+ \lambda \mu_2(t) e^{-\lambda \tau(t)}, \\
g_1:=&\ \rho f_2+\lambda \rho  f_1+ \mu_1(t)f_1+\mu_2(t)f_1e^{-\lambda \tau(t)}- \mu_2(t)\tau(t)e^{-\lambda \tau(t)} \int_0^1 f_5(x,s)e^{\lambda \tau(t)s}\,ds, \\
g_2:=&\ \mu f_4+ \lambda \mu f_3.
\end{aligned}
\end{eqnarray}
In order to solve \eqref{3:17:12}, we use a standard procedure, considering bilinear form $\Upsilon:((H_*(0,L)\times H_*(0,L))^2\to\mathds{R}$, given by
\begin{eqnarray}
\begin{aligned}
\Upsilon ((v,p),(\widetilde{v}, \widetilde{p})) = & \ \eta \int_0^L v\tilde{v}dx + \alpha \int_0^L v_x \tilde{v}_x\,dx - \gamma \beta \int_0^L p_x \tilde{v}_x\,dx \\
& + \lambda^2 \mu \int_0^L p \tilde{p}\,dx + \beta \int_0^L p_x \tilde{p}_x\,dx - \gamma \beta \int_0^L v_x \tilde{p}_x\,dx.
\end{aligned}
\end{eqnarray}
It is not difficult to show that $\Upsilon$ is continuous and coercive, so by applying the Lax-Milgram's Theorem, we obtain a solution for $(v,p)\in H_*(0,L)\times H_*(0,L)$ for \eqref{3:17:12}. In addition, it follows from \eqref{3:17:2} and \eqref{3:17:4} that $v,p\in H^2(0,L)$ and so $(v,u,p,q,z)\in D(\pazocal{A}(t))$.

Therefore, the operator $\lambda I - \pazocal{A}(t)$ is surjective for all $t>0$. Since $\kappa(t)>0$, we have
\begin{equation}
\lambda I - \widetilde{\pazocal{A}}(t) = (\lambda + \kappa(t))I - \pazocal{A}(t)\quad \text{is surjective}\quad\forall  t>0.\label{3:17:13:1}
\end{equation}
To complete the proof of {\bf (iii)}, it's suffices to prove that
\begin{eqnarray}
\frac{\| \Phi \|_t}{\| \Phi \|_s} \leq e^{\frac{c}{2\tau_0}|t-s|}, \quad \forall  t,s \in [0,T],\label{3:17:13}
\end{eqnarray}
where $\Phi = (v,u,p,q,z)^T$, $c$ is a positive constant and $\|\cdot\|_t$ is the norm associated to the inner product \eqref{3:7}. For $t,s \in [0,T]$, we have
\begin{eqnarray}
\begin{aligned}
\| \Phi \|_t^2 - \| \Phi \|_s^2 e^{\frac{c}{\tau_0}|t-s|} = & \left( 1 - e^{\frac{c}{\tau_0}|t-s|} \right) \int_0^L \left[ \rho u^2+ \mu q^2+ \alpha_1 v_x^2+ \beta (\gamma v_x - p_x )^2 \right]\,dx \\
& + \left( \xi(t)\tau(t) - \xi(s)\tau(s)e^{\frac{c}{\tau_0}|t-s|} \right) \int_0^L \int_0^1 z^2(x,y)\,dy\,dx.
\end{aligned}
\end{eqnarray}
It is clear that $1 - e^{\frac{c}{\tau_0}|t-s|} \leq 0$. Now we will prove $\xi(t)\tau(t) - \xi(s)\tau(s)e^{\frac{c}{\tau_0}|t-s|} \leq 0$ for some $c>0$. In order to do this, from \eqref{2:1} and MVT, we have
\begin{eqnarray}
\tau(t) = \tau(s) + \tau'(r)(t-s),
\end{eqnarray}
for some $r \in (s,t)$. Since $\xi$ is a non increasing function and $\xi>0$, we get
\begin{eqnarray}
\xi(t)\tau(t) \leq \xi(s)\tau(s) + \xi(s)\tau'(r)(t-s),
\end{eqnarray}
which implies
\begin{eqnarray}
\frac{\xi(t)\tau(t)}{\xi(s)\tau(s)} \leq 1 + \dfrac{|\tau'(r)|}{\tau(s)}|t-s|.
\end{eqnarray}
Using \eqref{2:1} and that $\tau'$ is bounded, we deduce that
\begin{eqnarray}
\frac{\xi(t)\tau(t)}{\xi(s)\tau(s)} \leq 1 + \frac{c}{\tau_0}|t-s| \leq e^{\frac{c}{\tau_0}|t-s|},
\end{eqnarray}
which proves \eqref{3:17:13} and therefore {\bf (iii)} follows.\\

\hspace{\oddsidemargin}{\bf (iv)} Note that, from {\bf (A1)}, we have

\begin{eqnarray}
\kappa'(t) = \frac{\tau'(t)\tau''(t)}{2\tau(t)\sqrt{1+\tau'(t)^2}} - \frac{\tau'(t)\sqrt{1+\tau'(t)^2}}{2\tau(t)^2}
\end{eqnarray}
is bounded on $[0,T]$ for all $T>0$. Moreover
\begin{eqnarray}
\frac{d}{dt}\pazocal{A}(t)U = \left( \begin{array}{c}
0 \\
-\rho^{-1}[\mu_1'(t)u+\mu_2'(t)z(\cdot,1)] \\
0 \\
0 \\
\frac{\tau''(t)\tau(t)y-\tau'(t)(\tau'(t)y-1)}{\tau(t)^2}z_y
\end{array} \right),
\end{eqnarray}
Since $\frac{\tau''(t)\tau(t)\rho-\tau'(t)(\tau'(t)\rho-1)}{\tau(t)^2}$ is  bounded on $[0,T]$ by {\bf (A1)}, and considering  {\bf (A2)} and {\bf (A3)}, we have
\begin{eqnarray}
\frac{d}{dt}\tilde{\pazocal{A}}(t) \in L_{*}^{\infty}([0,T], B(D(\mathcal{A}(0)), \mathcal{H})), \label{3:17:15}
\end{eqnarray}
where $L_{*}^{\infty}([0,T], B(D(\mathcal{A}(0)), \mathcal{H}))$ is the space of equivalence classes of essentially bounded, strongly measurable functions from $[0,T]$ into $B(D(\mathcal{A}(0)), \mathcal{H})$.

Then, \eqref{3:16}, \eqref{3:17:13:1} and \eqref{3:17:13} imply that the family $\widetilde{\pazocal{A}} = \{ \widetilde{\pazocal{A}}(t): t \in [0,T] \}$ is a stable family of generators in $\mathcal{H}$ with stability constants independent of $t$, by Proposition $1.1$ from \cite{isnb:978-3-642-11105-1-KATO}. Therefore, the assumptions (i)-(iv) of Theorem \ref{thm:3:1} are verified. Thus, the problem
\begin{equation}
\left\{\begin{array}{rcl}
\widetilde{U}_{t} &=& \widetilde{\pazocal{A}}(t) \widetilde{U},\\
\widetilde{U}(0) &=& U_0
\end{array}
\right.\quad
\end{equation}
has a unique solution $\widetilde{U} \in C([0,+\infty), D(\pazocal{A}(0)) ) \cap C^1([0,+\infty), \pazocal{H} )$ for $U_0 \in D(\pazocal{A}(0))$.
The requested solution of \eqref{pc} is then given by
\begin{eqnarray}
U(t) = e^{\int_0^t \kappa(s)ds}\widetilde{U}(t),
\end{eqnarray}
because
\begin{eqnarray}
\begin{aligned}
U_t(t) =&\ \kappa(t)e^{\int_0^t \kappa(s)\,ds}\widetilde{U}(t) + e^{\int_0^t \kappa(s)\,ds}\widetilde{U}_t(t) \\
=&\ e^{\int_0^t \kappa(s)\,ds} (\kappa(t) + \tilde{\pazocal{A}}(t) ) \widetilde{U}(t) \\
=&\ \pazocal{A}(t)e^{\int_0^t \kappa(s)\,ds}\widetilde{U}(t) \\
=&\ \pazocal{A}(t) U(t)
\end{aligned}
\end{eqnarray}
which concludes the proof.
\end{proof}

\section{Exponential stability}\label{sec:4}

This section is dedicated to study of the asymptotic behavior. We show that the solution of problem \eqref{2:9}-\eqref{2:11} is exponentially stable using the multiplier technique.

We define the energy associated to the  solution $U(t)=(v(t),v_t(t),p(t),p_t(t),z(t))$ of problem \eqref{2:9}-\eqref{2:11} by the following formula
\begin{eqnarray}
E(t) &= & \frac{1}{2} \int_0^L \left[ \rho v_t^2 + \mu p_t^2 + \alpha_1 v_x^2 + \beta( \gamma v_x - p_x)^2 \right]\,dx + \frac{\xi(t)\tau(t)}{2} \int_0^L \int_0^1 z^2\,dy\,dx.\label{4:1}
\end{eqnarray}
Our effort consists in building a suitable Lyapunov functional by the energy method. The main goal in this section is to prove the following stability result.

\begin{thm}\label{thm:4:1}
Let $U(t)=(v(t),v_t(t),p(t),p_t(t),z(t))$ be the solution of \eqref{2:9}-\eqref{2:11} with initial data $U_0\in D(\pazocal{A}(0))$ and $E(t)$ the energy of $U$. Then there exist positive constants $M$ and $\gamma$ such that
\begin{eqnarray}
E(t) \leq ME(0)e^{-\gamma t}, \quad \forall t \geq 0.\label{4:2}
\end{eqnarray}
\end{thm}

For the proof of Theorem \ref{thm:4:1} we need several lemmas. Our first result states that the energy is a non-increasing function and uniformly bounded above by $E(0)$.

\begin{lem}\label{lem:4:1}
Let $U(t)=(v(t),v_t(t),p(t),p_t(t),z(t))$ be the solution of \eqref{2:9}-\eqref{2:11}. Then the energy $E(t)$ satisfies
\begin{eqnarray}
\begin{aligned}
\frac{d}{dt} E(t) \leq &\ - \mu_1(t) \left( 1-\frac{\overline{\xi}}{2}- \frac{\delta}{2\sqrt{1-d}} \right) \int_0^L v_t^2\,dx \\
&\ - \mu_1(t) \left( \frac{\overline{\xi}(1-\tau'(t))}{2}- \frac{\delta \sqrt{1-d}}{2} \right) \int_0^L z^2(x,1)\,dx \leq 0.
\end{aligned}\label{4:3}
\end{eqnarray}
\end{lem}

\begin{proof}
Multiplying $\eqref{2:9}_1$ by $v_t$, $\eqref{2:9}_2$ by $p_t$ and integrating each of them by parts over $[0, L] $, we get
\begin{eqnarray}\label{y_1}
\begin{aligned}
 \frac{1}{2} \frac{d}{dt} \int_0^L ( \rho v_t^2 + \alpha_1 v_x^2 )\,dx + \gamma \beta \int_0^L (\gamma v_x - p_x ) v_{xt}\,dx + \mu_1(t)\int_0^L v_t^2\,dx + \mu_2(t)\int_0^L z(x,1)v_t\,dx=0,
\end{aligned}\label{4:4}\\
\begin{aligned}
\frac{1}{2} \frac{d}{dt} \int_0^L \mu p_t^2\,dx - \beta \int_0^L \left(\gamma v_x - p_x \right) p_{xt}\,dx = 0.\label{4:5}
\end{aligned}
\end{eqnarray}
Now multiplying $\eqref{2:9}_3$ by $\xi(t)z$ and integrating over $[0,L] \times [0,1]$, we obtain
\begin{eqnarray}
\frac{\tau(t)\xi(t)}{2}\int_0^L \int_0^1 \frac{d}{dt} z^2\,dy\,dx+\frac{\xi(t)}{2} \int_0^L \int_0^1 (1- \tau'(t)y)\frac{\partial }{\partial y}z^2\,dy\,dx = 0,
\end{eqnarray}
which is equivalent to
\begin{eqnarray}
\begin{aligned}
\frac{d}{dt} \left( \frac{\xi(t)\tau(t)}{2} \int_0^L \int_0^1 z^2\,dy\,dx \right) = &\ \frac{\xi(t)}{2} \int_0^L v_t^2\,dx - \frac{\xi(t)}{2} \int_0^L z^2(x,1)\,dx + \frac{\xi(t)\tau'(t)}{2} \int_0^L z^2(x,1)\,dx \\
 &\ + \frac{\xi'(t)\tau(t)}{2} \int_0^L \int_0^1 z^2\,dy\,dx.
\end{aligned}\label{4:6}
\end{eqnarray}
Combining \eqref{4:4}, \eqref{4:5} and \eqref{4:6}, we obtain
\begin{eqnarray}
\begin{aligned}
\frac{d}{dt} E(t) = & - \mu_1(t)\int_0^L v_t^2\,dx - \mu_2(t)\int_0^L z(x,1) v_t\,dx + \frac{\xi(t)}{2} \int_0^L v_t^2\,dx - \frac{\xi(t)}{2} \int_0^L z^2(x,1)\,dx \\
& + \frac{\xi(t)\tau'(t)}{2} \int_0^L z^2(x,1)\,dx + \frac{\xi'(t)\tau(t)}{2} \int_0^L \int_0^1 z^2\,dy\,dx.
\end{aligned}
\end{eqnarray}
Applying Young's inequality and taking into account \eqref{3:8}, {\bf(A2)} (which results in $\xi'(t)\leq 0$), we have
\begin{eqnarray}
\begin{aligned}
\frac{d}{dt} E(t) \leq &\  -\left( \mu_1(t) - \frac{\xi(t)}{2} - \frac{|\mu_2(t) |}{2\sqrt{1-d}} \right) \int_0^L v_t^2\,dx \\
&\ - \left( \frac{\xi(t)}{2} - \frac{\xi(t)\tau'(t)}{2} - \frac{|\mu_2(t)| \sqrt{1-d}}{2} \right) \int_0^L z^2(x,1)\,dx \\
&\ + \frac{\xi'(t)\tau(t)}{2} \int_0^L \int_0^1 z^2\,dy\,dx \\
\leq &\ -\mu_1(t) \left( 1-\frac{\bar{\xi}}{2}- \frac{\delta}{2\sqrt{1-d}} \right) \int_0^L v_t^2\,dx \\
&\ - \mu_1(t) \left( \frac{\bar{\xi}(1-\tau'(t))}{2}- \frac{\delta \sqrt{1-d}}{2} \right) \int_0^L z^2(x,1)\,dx \leq 0.
\end{aligned}
\end{eqnarray}
Hence, the proof is complete.
\end{proof}

In the previous result we observe that the energy functional restores some energy terms with a negative sign. We are interested in building a Lyapunov functional that restores the full energy of the system with negative sign, and for this goal, we consider the following lemmas.

\begin{lem}\label{lem:4:2}
If $U(t)=(v(t),v_t(t),p(t),p_t(t),z(t))$ is a solution of \eqref{2:9}-\eqref{2:11}, then the functional $I_1$, defined by
\begin{eqnarray}
I_1(t) = \rho \int_0^L v v_{t}\,dx + \gamma \mu \int_0^L v p_{t}\,dx\label{4:7}
\end{eqnarray}
satisfies the estimative
\begin{eqnarray}
\frac{d}{dt} I_1(t) \leq -\frac{\alpha_1}{2} \int_0^L v_x^2\,dx + \varepsilon_1 \int_0^L p_t^2\,dx + c_1 \int_0^L z^2(x,1)\,dx + c_1 \left( 1 + \frac{1}{\varepsilon_1} \right) \int_0^L v_t^2\,dx,\label{4:8}
\end{eqnarray}
for any constants $\varepsilon_1 > 0$ and $c_1 > 0$.
\end{lem}

\begin{proof}
Taking derivative of $I_1(t)$, using \eqref{2:9} and integrating by parts, we arrive at
	
\begin{eqnarray}
\begin{aligned}
\frac{d}{dt}I_1(t) = &\ -\alpha_1 \int_0^L v_x^2\,dx + \rho \int_0^L v_t^2\,dx + \gamma \mu \int_0^L p_t v_t\,dx \\
&\ - \mu_1(t) \int_0^L v_t v\,dx - \mu_2(t) \int_0^L z(x,1) v\,dx.
\end{aligned}
\end{eqnarray}
From {\bf(A2)} and {\bf(A3)}, we have
\begin{eqnarray}
\begin{aligned}
\frac{d}{dt}I_1(t) = &\ -\alpha_1 \int_0^L v_x^2\,dx + \rho \int_0^L v_t^2\,dx + \gamma \mu \int_0^L p_t v_t\,dx \\
&\ + \mu_1(0) \int_0^L |v_t v|\,dx + \delta \mu_1(0) \int_0^L |z(x,1)v|\,dx
\end{aligned}
\end{eqnarray}
Estimate \eqref{4:8} follows thanks to Young's and Poincar\'e's inequalities.
\end{proof}

\begin{lem}\label{lem:4:3}
If $U(t)=(v(t),v_t(t),p(t),p_t(t),z(t))$ is a solution of \eqref{2:9}-\eqref{2:11}, then the functional $I_2$, defined by
\begin{eqnarray}
I_2(t) = \rho \int_0^L v_t(\gamma v - p)\,dx + \gamma \mu \int_0^L p_t (\gamma v - p )\,dx\label{4:9}
\end{eqnarray}
satisfies the estimative
\begin{eqnarray}
\begin{aligned}
\frac{d}{dt} I_2(t) \leq &\  - \frac{\gamma \mu}{2} \int_0^L p_t^2\,dx + 3\alpha_1 \varepsilon_2 \int_0^L (\gamma v_x - p_x)^2\,dx +  \frac{c_2}{\varepsilon_2} \int_0^L v_x^2\,dx \\
&\ +\frac{c_2}{\varepsilon_2} \int_0^L z^2(x,1)\,dx + c_2 \left( 1 + \frac{1}{\varepsilon_2} \right) \int_0^L v_t^2\,dx
\end{aligned}\label{4:10}
\end{eqnarray}
for any constants $\varepsilon_2 > 0$ and $c_2 > 0$.
\end{lem}

\begin{proof}
Taking derivative of $I_2(t)$, using \eqref{2:9} together with integration by parts, we obtain
\begin{eqnarray}
\begin{aligned}
\frac{d}{dt}I_2(t) = &\ - \alpha_1 \int_0^L v_x ( \gamma v_x - p_x )\,dx + \rho \gamma \int_0^L v_t^2\,dx - \rho \int_0^L v_t p_t\,dx + \gamma^2 \mu \int_0^L p_t v_t\,dx \\
&\ - \gamma \mu \int_0^L p_t^2\,dx - \mu_1(t) \int_0^L v_t (\gamma v - p)\,dx - \mu_2(t) \int_0^L z(x,1) (\gamma v - p)\,dx.
\end{aligned}
\end{eqnarray}
From {\bf(A2)} and {\bf(A3)}, we obtain
\begin{eqnarray}
\begin{aligned}
\frac{d}{dt}I_2(t) = & - \alpha_1 \int_0^L v_x (\gamma v_x - p_x)\,dx + \rho \gamma \int_0^L v_t^2\,dx - \rho \int_0^L v_t p_t\,dx + \gamma^2 \mu \int_0^L p_t v_t\,dx \\
&\ - \gamma \mu \int_0^L p_t^2\,dx + \mu_1(0) \int_0^L |v_t (\gamma v-p)|\,dx + \delta \mu_1(0) \int_0^L |z(x,1)(\gamma v-p)|\,dx.
\end{aligned}
\end{eqnarray}
We then use Young's and Poincar\'e's inequalities to obtain \eqref{4:10}.
\end{proof}

\begin{lem}\label{lem:4:4}
If $U(t)=(v(t),v_t(t),p(t),p_t(t),z(t))$ is a solution of \eqref{2:9}-\eqref{2:11}, then the functional $I_3$, defined by
\begin{eqnarray}
I_3(t) = \rho \int_0^L v_t v\,dx + \mu \int_0^L p_t p\,dx\label{4:11}
\end{eqnarray}
satisfies the estimative
\begin{eqnarray}
\begin{aligned}
\frac{d}{dt} I_3(t) \leq & - \frac{\alpha_1}{2}\int_0^L v_x^2\,dx - \beta \int_0^L (\gamma v_x - p_x)^2\,dx + \mu \int_0^L p_t^2\,dx \\
&\ + c_3 \int_0^L v_t^2\,dx + c_3 \int_0^L z^2(x,1)\,dx,
\end{aligned}\label{4:12}
\end{eqnarray}
for any constant $c_3 > 0$.
\end{lem}

\begin{proof}
Taking derivative of $I_3(t)$, use \eqref{2:9} and integrating by parts, yield
\begin{eqnarray}
\begin{aligned}
\frac{d}{dt} I_3(t) \leq &\  - \alpha_1 \int_0^L v_x^2\,dx - \beta \gamma^2 \int_0^L v_x^2\,dx + 2 \beta \gamma \int_0^L p_x v_x\,dx - \beta \int_0^L p_x^2\,dx+\rho \int_0^L v_t^2\,dx \\
&\ + \mu \int_0^L p_t^2\,dx - \mu_1(t) \int_0^L v_t v\,dx - \mu_2(t) \int_0^L z(x,1) v\,dx.
\end{aligned}
\end{eqnarray}
From {\bf(A2)} and {\bf(A3)}, and taking into account
\begin{eqnarray}
-\beta(\gamma v_x-p_x)^2 = -\beta \gamma^2 v_x^2 + 2\beta \gamma v_x p_x - \beta p_x^2,
\end{eqnarray}
we have
\begin{eqnarray}
\begin{aligned}
\frac{d}{dt} I_3(t) \leq &\  - \alpha_1 \int_0^L v_x^2\,dx - \beta \int_0^L (\gamma v_x - p_x)^2\,dx + \rho \int_0^L v_t^2\,dx + \mu \int_0^L p_t^2\,dx \\
&\ + \mu_1(0) \int_0^L|v_t v|\,dx + \delta \mu_1(0) \int_0^L |z(x,1) v|\,dx.
\end{aligned}
\end{eqnarray}
	Exploiting Young's and Poincar\'e's inequalities, we obtain the estimates \eqref{4:12} and conclude the prove.
\end{proof}

As in \cite{doi:10.3934/cpaa.2011.10.667}, taking into account the last lemma, we introduce the functional
\begin{equation}
J(t) = \overline{\xi} \tau(t) \int_0^L \int_0^1 e^{-2\tau(t)y} z^2(x,y)\,dy\,dx.\label{4:13}
\end{equation}
For this functional we have the following estimate.

\begin{lem}[{\cite[Lemma 3.7]{doi:10.3934/cpaa.2011.10.667}}]\label{lem:4:5}
Let $U(t)=(v(t),v_t(t),p(t),p_t(t),z(t))$ be solution of \eqref{2:9}-\eqref{2:11}. Then the functional $J(t)$ satisfies
\begin{equation}
\frac{d}{dt}J(t) \leq -2 J(t) + \overline{\xi} \int_0^L v_t^2\,dx.\label{4:14}
\end{equation}
\end{lem}

Now we are in position to prove our principal result.
\begin{proof}[Proof of Theorem \ref{thm:4:1}]
We will to construct a suitable Lyapunov functional $\pazocal{L}$ satisfying the following equivalence relation
\begin{eqnarray}
\gamma_1 E(t) \leq \pazocal{L}(t) \leq \gamma_2 E(t), \quad \forall t \geq 0,
\end{eqnarray}
for some $\gamma_1, \gamma_2 > 0$ and to prove that
\begin{eqnarray}
\frac{d}{dt}\mathcal{L}(t) \leq -\lambda \mathcal{L}(t), \quad \forall t \geq 0,
\end{eqnarray}
for some $\lambda>0$, which implies
\begin{eqnarray}
\mathcal{L}(t) \leq \mathcal{L}(0)e^{-\lambda t}, \quad \forall t \geq 0.
\end{eqnarray}	
Let us define the Lyapunov functional
\begin{eqnarray}
\mathcal{L}(t) = N E(t)(t) + \sum_{i=1}^{3}N_i I_i(t) + J(t),\label{4:15}
\end{eqnarray}
where $N_i$, $i=1,2,3$ are positive real numbers which will be chosen later. By the Lemma \ref{lem:4:1}, there exists a positive constant $K$ such that
\begin{eqnarray}
\frac{d}{dt}E(t) \leq -K \left( \int_0^L v_t^2\,dx + \int_0^L z^2(x,1)\,dx \right).\label{4:16}
\end{eqnarray}
We have that
\begin{eqnarray}
\begin{aligned}
| \mathcal{L}(t) - NE(t)| \leq &\ N_1 \left( \rho \int_0^L|v_t v|\,dx + \gamma \mu \int_0^L |p_t v|\,dx\right) \\
&\ + N_2 \left( \rho \int_0^L |v_t ( \gamma v - p )|\,dx + \gamma \mu \int_0^L |p_t (\gamma v - p)|\,dx \right) \\
&\ + N_3 \left( \rho \int_0^L |v_t v |\,dx + \mu \int_0^L |p_t p |\,dx \right) \\
&\ + \left| \overline{\xi}\tau(t) \int_0^L \int_0^1 e^{-2\tau(t)y} z^2\,dy\,dx \right|.
\end{aligned}
\end{eqnarray}
It follows from \eqref{4:1}, Young's and Poincar\'e's inequalities and from the fact that $\tau(t) \leq \tau_1$ for all $t\geq 0$ and $e^{-2\tau(t)y} \leq 1$ for all $y \in (0,1)$ that
\begin{eqnarray}
|\mathcal{L}(t) - NE(t)| \leq \gamma_3 \int_0^L \left[ v_t^2 + p_t^2 + v_x^2 + ( \gamma v_x - p_x)^2 + \int_0^1 z^2\,dy \right]\,dx \leq \gamma_3 E(t)
\end{eqnarray}
for some constant $\gamma_3 > 0$. So, we can choose $N$ large enough that $\gamma_1 := N - \gamma_3$ and $\gamma_2 := N + \gamma_3$, then
\begin{eqnarray}
\gamma_1 E(t) \leq \mathcal{L}(t) \leq \gamma_2 E(t), \quad \forall t \geq 0\label{4:17}
\end{eqnarray}
holds.

Now, taking derivative $\pazocal{L}(t)$, substitute the estimates \eqref{4:8}, \eqref{4:10}, \eqref{4:12}, \eqref{4:13}, \eqref{4:16} and setting
\begin{eqnarray}
N_2 = \frac{8}{\gamma}, \quad N_3=1, \quad \varepsilon_1 = \frac{\mu}{N_1} \quad \mbox{and} \quad \varepsilon_2 = \frac{\beta \gamma}{48 \alpha_1},
\end{eqnarray}
we obtain that
\begin{eqnarray}
\begin{aligned}
\frac{d}{dt}\pazocal{L}(t) \leq &\  -\left[ NK - c_1 \left( 1 + \frac{N_1}{\mu} \right)N_1 - c_2 \left( 1 + \frac{48\alpha_1}{\beta \gamma} \right) \frac{8}{\gamma} - c_3 - \overline{\xi} \right] \int_0^L v_t^2\,dx \\
&\  - \left( NK - c_1 N_1 - \frac{384 \alpha_1 c_2}{\beta \gamma^2} - c_3 \right) \int_0^L z^2(x,1)\,dx \\
& - \left[ \frac{\alpha_1}{2}N_1 - \frac{384 \alpha_1 c_2}{\beta \gamma^2} + \frac{\alpha_1}{2} \right] \int_0^L v_x^2\,dx \\
& - 2\mu \int_0^L p_t^2\,dx - \frac{\beta}{2} \int_0^L (\gamma v_x - p_x)^2\,dx - 2 J(t).
\end{aligned}\label{4:18}
\end{eqnarray}
First, let us choose $N_1$ large enough such that
\begin{eqnarray}
\frac{\alpha_1}{2}N_1 - \frac{384 \alpha_1 c_2}{\beta \gamma^2} + \frac{\alpha_1}{2} > 0.
\end{eqnarray}
Now, since $\xi(t)\tau(t)$ non-negative and limited and choosing $N$ large enough that \eqref{4:18} is taken into the following estimate
\begin{eqnarray}
\begin{aligned}
\frac{d}{dt}\pazocal{L}(t) \leq & - \eta \int_0^L \left[ v_t^2 + p_t^2 + v_x^2 +(\gamma v_x - p_x)^2 + z^2(x,1) + \int_0^1 z^2\,dy \right]\,dx \\
\leq &\  -\eta \int_0^L \left[ v_t^2 + p_t^2 + v_x^2 +(\gamma v_x - p_x)^2 + \int_0^1 z^2\,dy \right]\,dx,
\end{aligned}
\end{eqnarray}
for some positive constant $\eta$.
Therefore, from \eqref{4:1}, we have
\begin{equation}
\frac{d}{dt}\mathcal{L}(t) \leq -\eta E(t), \quad \forall t > 0.\label{4:19}
\end{equation}
In view of \eqref{4:17} and \eqref{4:19}, we note that
\begin{eqnarray}
\frac{d}{dt}\mathcal{L}(t) \leq -\lambda \mathcal{L}(t),\quad  \forall t > 0,\label{4:20}
\end{eqnarray}
which leads to
\begin{equation}
\mathcal{L}(t) \leq \mathcal{L}(0)e^{-\lambda t}, \forall t > 0.\label{4:21}
\end{equation}
The desired result \eqref{4:2} follows by using estimates \eqref{4:17} and \eqref{4:21}. Then, the proof of Theorem \ref{thm:4:1} is complete.
\end{proof}

\bibliographystyle{plain}
\bibliography{biblio}

\begin{thebibliography}{10}

\bibitem{doi:10.1002/(SICI)1096-9845(199711)26:11<1169::AID-EQE702>3.0.CO;2-S}
A.~K. Agrawal and J.~N. Yang.
\newblock Effect of fixed time delay on stability and performance of actively
  controlled civil engineering structures.
\newblock {\em Earthquake Engineering \& Structural Dynamics},
  26(11):1169--1185, 1997.

\bibitem{doi:10.1080/00036811.2014.1000314}
T.~A. Apalara.
\newblock Asymptotic behavior of weakly dissipative {T}imoshenko system with
  internal constant delay feedbacks.
\newblock {\em Applicable Analysis}, 95(1):187--202, 2016.

\bibitem{isbn:9780471970248}
H.~T. Banks, R.~C. Smith, and Y.~Wang.
\newblock {\em Smart Material Structures: Modeling, Estimation and Control}.
\newblock Wiley-Masson Series Research in Applied Mathematics. Wiley, 1996.

\bibitem{10.3934/era.2020014}
V.~Barros, C.~Nonato, and C.~Raposo.
\newblock Global existence and energy decay of solutions for a wave equation
  with non-constant delay and nonlinear weights.
\newblock {\em Electronic Research Archive}, 28(2688-1594_2020_28_1_205):205,
  2020.

\bibitem{doi:10.14232/ejqtde.2014.1.11}
A.~Benaissa, A.~Benguessoum, and S.~A. Messaoudi.
\newblock Energy decay of solutions for a wave equation with a constant weak
  delay and a weak internal feedback.
\newblock {\em Electron. J. Qual. Theory Differ. Equ.}, 2014(11):1--13, 2014.

\bibitem{doi:10.1007/978-1-4939-2107-2}
D.~Breda, S.~Maset, and R.~Vermiglio.
\newblock {\em Stability of Linear Delay Differential Equations: A Numerical
  Approach with MATLAB}.
\newblock SpringerBriefs in Control, Automation and Robotics. Springer-Verlag
  New York, 2015.

\bibitem{doi:10.1137/0326040}
R.~Datko.
\newblock Not all feedback stabilized hyperbolic systems are robust with
  respect to small time delays in their feedbacks.
\newblock {\em SIAM Journal on Control and Optimization}, 26(3):697--713, 1988.

\bibitem{doi:10.1109/CDC.1985.268529}
R.~Datko, J.~Lagnese, and M.~P. Polis.
\newblock An example of the effect of time delays in boundary feedback
  stabilization of wave equations.
\newblock In {\em 1985 24th IEEE Conference on Decision and Control}, pages
  117--117, 1985.

\bibitem{DESTUYNDER-1992}
Ph. Destuynder, L.~Legrain, I.~Castel, and N.~Richard.
\newblock Theoretical, numerical and experimental discussion on the use of
  piezoelectric devices for control-structure interaction.
\newblock {\em European J. Mech. A Solids}, 11(2):181--213, 1992.

\bibitem{doi:10.3934/dcdsb.2020206}
M.~J. {Dos Santos}, B.~Feng, D.~S. {Almeida Júnior}, and M.~L. Santos.
\newblock Global and exponential attractors for a nonlinear porous elastic
  system with delay term.
\newblock {\em Discrete \& Continuous Dynamical Systems - B}, 22(11):1--24,
  2020.

\bibitem{doi:10.1063/5.0006680}
M.~J. {Dos Santos}, M.~M. Freitas, A.~J.~A. Ramos, D.~S. {Almeida Júnior}, and
  L.~R.~S. Rodrigues.
\newblock Long-time dynamics of a nonlinear {T}imoshenko beam with discrete
  delay term and nonlinear damping.
\newblock {\em J. Math. Phys.}, 61(061505):1--17, 2020.

\bibitem{doi:10.1186/s13661-015-0468-4}
B.~Feng and M.~L. Pelicer.
\newblock Global existence and exponential stability for a nonlinear
  {T}imoshenko system with delay.
\newblock {\em Boundary Value Problems}, 2015(1):206, 2015.

\bibitem{doi:10.1080/00036811.2016.1148139}
B.~Feng and X.~Yang.
\newblock Long-time dynamics for a nonlinear {T}imoshenko system with delay.
\newblock {\em Applicable Analysis}, 96(4):606--625, 2017.

\bibitem{doi:10.1109/CDC.1998.757931}
S.~{Hansen}.
\newblock Analysis of a plate with a localized piezoelectric patch.
\newblock In {\em Proceedings of the 37th IEEE Conference on Decision and
  Control (Cat. No.98CH36171)}, volume~3, pages 2952--2957 vol.3, 1998.

\bibitem{doi:10.1007/978-3-662-05030-9}
H.~Y. Hu and Z.~H. Wang.
\newblock {\em Dynamics of Controlled Mechanical Systems with Delayed
  Feedback}.
\newblock Springer-Verlag Berlin Heidelberg, 2002.

\bibitem{doi:10.1137/050629884}
B.~Kapitonov, B.~Miara, and G.~Perla Menzala.
\newblock Boundary observation and exact control of a quasi-electrostatic
  piezoelectric system in multilayered media.
\newblock {\em SIAM Journal on Control and Optimization}, 46(3):1080--1097,
  2007.

\bibitem{isbn:978-88-7642-248-5}
T.~Kato.
\newblock {\em Abstract differential equations and nonlinear mixed problems}.
\newblock Publications of the Scuola Normale Superiore. Edizioni della Normale,
  1988.

\bibitem{isnb:978-3-642-11105-1-KATO}
T.~Kato.
\newblock Linear and quasi-linear equations of evolution of hyperbolic type.
\newblock In Giuseppe~Da Prato and Giuseppe Geymonat, editors, {\em
  Hyperbolicity: Lectures given at a Summer School of the Centro Internazionale
  Matematico Estivo (C.I.M.E.) held in Cortona (Arezzo), Italy, June 24 - July
  2, 1976}, volume~72 of {\em C.I.M.E. Summer Schools}. Springer-Verlag Berlin
  Heidelberg, 2011.

\bibitem{doi:10.1080/15376490802666310}
Ö. Kayacik, J.~C. {Bruch Jr.}, J.~M. Sloss, S.~Adali, and I.~S. Sadek.
\newblock Piezo control of free vibrations of damped beams with time delay in
  the sensor feedback.
\newblock {\em Mechanics of Advanced Materials and Structures}, 16(5):345--355,
  2009.

\bibitem{doi:10.1007/s00033-011-0145-0}
M.~Kirane and B.~{Said-Houari}.
\newblock Existence and asymptotic stability of a viscoelastic wave equation
  with a delay.
\newblock {\em Zeitschrift f{\"u}r angewandte Mathematik und Physik},
  62(6):1065--1082, 2011.

\bibitem{doi:10.3934/cpaa.2011.10.667}
M.~Kirane, B.~{Said-Houari}, and M.~N. Anwar.
\newblock Stability result for the timoshenko system with a time-varying delay
  term in the internal feedbacks.
\newblock {\em Communications on Pure {\&} Applied Analysis}, 10(2):667, 2011.

\bibitem{doi:doi.org/10.1016/j.crma.2008.12.007}
I.~Lasiecka and B.~Miara.
\newblock Exact controllability of a 3d piezoelectric body.
\newblock {\em Comptes Rendus Mathematique}, 347(3):167--172, 2009.

\bibitem{doi:10.1137/1030001}
J.~L. Lions.
\newblock Exact controllability, stabilization and perturbations for
  distributed systems.
\newblock {\em SIAM Review}, 30(1):1--68, 1988.

\bibitem{doi:10.1177/0263092316628255}
C.~Liu, S.~Yue, and J.~Zhou.
\newblock Piezoelectric optimal delayed feedback control for nonlinear
  vibration of beams.
\newblock {\em Journal of Low Frequency Noise, Vibration and Active Control},
  35(1):25--38, 2016.

\bibitem{doi:10.1007/s00161-017-0556-z}
W.~Liu and M.~Chen.
\newblock Well-posedness and exponential decay for a porous thermoelastic
  system with second sound and a time-varying delay term in the internal
  feedback.
\newblock {\em Continuum Mechanics and Thermodynamics}, 29(3):731--746, 2017.

\bibitem{doi:10.1109/CDC.2013.6760341}
K.~{Morris} and A.~Ö. {Özer}.
\newblock Strong stabilization of piezoelectric beams with magnetic effects.
\newblock In {\em 52nd IEEE Conference on Decision and Control}, pages
  3014--3019, 2013.

\bibitem{doi:10.1137/130918319}
K.~A. Morris and A.~Ö. Özer.
\newblock Modeling and stabilizability of voltage-actuated piezoelectric beams
  with magnetic effects.
\newblock {\em SIAM Journal on Control and Optimization}, 52(4):2371--2398,
  2014.

\bibitem{doi:10.1137/060648891}
S.~Nicaise and C.~Pignotti.
\newblock Stability and instability results of the wave equation with a delay
  term in the boundary or internal feedbacks.
\newblock {\em SIAM Journal on Control and Optimization}, 45(5):1561--1585,
  2006.

\bibitem{NICAISE-PIGNOTTI-2011}
S.~Nicaise and C.~Pignotti.
\newblock Interior feedback stabilization of wave equations with time
  dependence delay.
\newblock {\em Electron. J. Diff. Equ.}, 2011(41):1--20, 2011.

\bibitem{doi:10.3934/dcdss.2011.4.693}
S.~Nicaise, C.~Pignotti, and J.~Valein.
\newblock Exponential stability of the wave equation with boundary time-varying
  delay.
\newblock {\em Discrete \& Continuous Dynamical Systems - S}, 4(3):693--722,
  2011.

\bibitem{doi:10.1007/s00033-018-0934-9}
S.~Park.
\newblock Long-time behavior for suspension bridge equations with time delay.
\newblock {\em Z. Angew. Math. Phys.}, 69(45):1--12, 2018.

\bibitem{doi:10.1007/978-1-4612-5561-1}
H.~Pazy.
\newblock {\em Semigroups of linear operators and applications to partial
  differential equations}.
\newblock Springer, New York, 1983.

\bibitem{doi:10.3390/app9081557}
J.~Peng, M.~Xiang, L.~Li, H.~Sun, and X.~Wang.
\newblock Time-delayed feedback control of piezoelectric elastic beams under
  superharmonic and subharmonic excitations.
\newblock {\em Applied Sciences}, 9(8), 2019.

\bibitem{doi:10.1088/1361-665x/ab2e3d}
P.~Peng, G.~Zhang, M.~Xiang, H.~Sun, X.~Wang, and X.~Xie.
\newblock Vibration control for the nonlinear resonant response of a
  piezoelectric elastic beam via time-delayed feedback.
\newblock {\em Smart Materials and Structures}, 28(9):095010, aug 2019.

\bibitem{doi:10.1063/1.1139566}
D.~W. Pohl.
\newblock Dynamic piezoelectric translation devices.
\newblock {\em Review of Scientific Instruments}, 58(1):54--57, 1987.

\bibitem{doi:10.1051/m2an/2018004}
A.~J.~A. Ramos, C.~S. L.~Gon\c calves, and S.~S. {Corr\^ea Neto}.
\newblock Exponential stability and numerical treatment for piezoelectric beams
  with magnetic effect.
\newblock {\em ESAIM. Mathematical Modelling and Numerical Analysis},
  52(1):255--274, 2018.

\bibitem{doi:10.1007/s10884-019-09799-2}
A.~J.~A. Ramos, M.~J. {Dos Santos}, M.~M. Freitas, and D.~S. {Almeida
  J{\'u}nior}.
\newblock Existence of attractors for a nonlinear {T}imoshenko system with
  delay.
\newblock {\em Journal of Dynamics and Differential Equations}, pages 1--24,
  2019.

\bibitem{doi:10.1007/s00033-019-1106-2}
A.~J.~A. Ramos, M.~M Freitas, D.~S. Almeida, and et~al.
\newblock Equivalence between exponential stabilization and boundary
  observability for piezoelectric beams with magnetic effect.
\newblock {\em Z. Angew. Math. Phys.}, 70(60):1--14, 2019.

\bibitem{doi:10.1016/j.jmaa.2018.06.017}
C.~A. Raposo, T.~A. Apalara, and J.~O. Ribeiro.
\newblock Analyticity to transmission problem with delay in porous-elasticity.
\newblock {\em Journal of Mathematical Analysis and Applications},
  466(1):819--834, 2018.

\bibitem{RAPOSO-CHUQUIPOMA-AVILA-SANTOS-2013}
C.~A. Raposo, J.~A.~D. Chuquipoma, J.~A.~J. Avila, and M.~L. Santos.
\newblock Exponential decay and numerical solution for a timoshenko system with
  delay term in the internal feedbacky.
\newblock {\em Int. J. Anal. Appl.}, 3(1):1--13, 2013.

\bibitem{isbn:9780849344596}
N.~N. Rogacheva.
\newblock {\em The Theory of Piezoelectric Shells and Plates}.
\newblock CRC Press, 1994.

\bibitem{doi:10.1016/j.amc.2010.08.021}
B.~Said-Houari and Y.~Laskri.
\newblock A stability result of a {T}imoshenko system with a delay term in the
  internal feedback.
\newblock {\em Applied Mathematics and Computation}, 217(6):2857--2869, 2010.

\bibitem{doi:10.1137/1.9780898717471}
R.~C. Smith.
\newblock {\em Smart Material Systems}.
\newblock Society for Industrial and Applied Mathematics, 2005.

\bibitem{doi:10.1007/s10444-004-7629-9}
L.~T. Tebou and E.~Zuazua.
\newblock Uniform boundary stabilization of the finite difference space
  discretization of the 1-d wave equation.
\newblock {\em Advances in Computational Mathematics}, 26(1):337, 2006.

\bibitem{doi:10.1007/978-1-4899-6453-3}
H.~F. Tiersten.
\newblock {\em Linear Piezoelectric Plate Vibrations}.
\newblock Springer, 1969.

\bibitem{isbn:9789402412581}
H.~Tzou.
\newblock {\em Piezoelectric Shells: Sensing, Energy Harvesting, and
  Distributed Control}.
\newblock Solid Mechanics and Its Applications. Springer Netherlands, 2^o
  edition, 2018.

\bibitem{doi:10.1016/B978-0-08-102135-4.00001-1}
K.~Uchino.
\newblock Chapter 1 - {T}he development of piezoelectric materials and the new
  perspective.
\newblock In Kenji Uchino, editor, {\em Advanced Piezoelectric Materials},
  Woodhead Publishing in Materials, pages 1--92. Woodhead Publishing, second
  edition, 2017.

\bibitem{doi:10.1186/s13660-019-2133-4}
S.~Wang and Q.~Ma.
\newblock Uniform attractors for the non-autonomous suspension bridge equation
  with time delay.
\newblock {\em Journal of Inequalities and Applications}, 2019(1):180, 2019.

\bibitem{doi:10.1051/cocv:2006021}
G.~Q. Xu, S.~P. Yung, and L.~K. Li.
\newblock Stabilization of wave systems with input delay in the boundary
  control.
\newblock {\em ESAIM: Control, Optimisations and Calculus of Variational},
  12(4):770--785, 2006.

\bibitem{doi:10.1007/b101799}
J.~Yang.
\newblock {\em An Introduction to the Theory of Piezoelectricity}.
\newblock Advances in Mechanics and Mathematics. Springer, 2005.

\bibitem{doi:10.1007/s00245-018-9539-0}
X.~Yang, J.~Zhang, and Y.~Lu.
\newblock Dynamics of the nonlinear {T}imoshenko system with variable delay.
\newblock {\em Applied Mathematics {\&} Optimization}, pages 1--30, 2018.

\bibitem{doi:10.1016/j.simpat.2007.11.005}
T.-J. Yeh, R.-F. Hung, and Lu~S.-W.
\newblock An integrated physical model that characterizes creep and hysteresis
  in piezoelectric actuators.
\newblock {\em Simulation Modelling Practice and Theory}, 16(1):93--110, 2008.

\bibitem{doi:10.1109/ACC.2014.6858862}
A.~{Özkan Özer} and K.~A. {Morris}.
\newblock Modeling an elastic beam with piezoelectric patches by including
  magnetic effects.
\newblock In {\em 2014 American Control Conference}, pages 1045--1050, 2014.

\end{thebibliography}

\end{document}